\documentclass[a4paper]{amsart}
\usepackage[utf8]{inputenc}
\usepackage[english]{babel}

\usepackage{amscd,amssymb}
\usepackage{amsmath}
\usepackage{amsthm}
\usepackage{graphicx}
\usepackage{fancybox}
\usepackage{epic,eepic}
\usepackage{amstext}
\usepackage{mathtools}
\usepackage{esint}
\usepackage[square,numbers]{natbib}
\usepackage{booktabs}
\usepackage[hide links]{hyperref}
\usepackage{comment}
\usepackage{mathtools}
\usepackage{tikz,pgfplots}
\usepackage{subcaption}

\usepackage{autonum}

\newtheorem{theorem}{Theorem}[section]
\newtheorem{proposition}[theorem]{Proposition}

\newtheorem{lemma}[theorem]{Lemma}

\theoremstyle{definition}
\newtheorem{definition}[theorem]{Definition}

\theoremstyle{remark}
\newtheorem{remark}[theorem]{Remark}
\numberwithin{theorem}{section}
\numberwithin{equation}{section}
\numberwithin{figure}{section}

\usepackage{xspace}
\newcommand{\Chebname}{Chebyshev\xspace}

\newcommand{\calK}{\mathcal{K}}

\newcommand{\calO}{\mathcal{O}}
\newcommand{\calP}{\mathcal{P}}

\newcommand{\frakL}{\mathfrak{L}}
\newcommand{\frakM}{\mathfrak{M}}

\newcommand{\Knz}[1]{\calK^{#1}}
\newcommand{\knz}[1]{k^{#1}}

\newcommand{\Sigman}[1]{\Sigma^{#1}}
\newcommand{\Sigmanll}{\Sigma^{\lambda,\Lambda}}

\newcommand{\numNZ}{\eta}
\newcommand{\numNZrowSymbol}{\eta}
\newcommand{\numNZrowi}[1]{\numNZrowSymbol_{#1}}
\newcommand{\numNZrowvector}{\boldsymbol{\numNZrowSymbol}} 
\newcommand{\numNZrowMax}{\numNZrowSymbol_\mathrm{max}}

\newcommand{\uppercondbound}{\widehat{\kappa}}
\newcommand{\mineig}{\lambda_{\min}}
\newcommand{\maxeig}{\lambda_{\max}}

\newcommand{\SparseSpacekk}{\mathcal{S}^{\numNZrowvector}}

\newcommand{\layers}{L}
\newcommand{\entries}{M}
\newcommand{\layersfun}{\frakL}
\newcommand{\entriesfun}{\frakM}
\newcommand{\Reals}{\mathbb{R}}
\newcommand{\NatNum}{\mathbb{N}}
\newcommand{\Id}{\mathrm{I}}
\newcommand{\rhs}{r}

\newcommand{\NNs}[0]{\Phi}
\newcommand{\weight}[0]{W}
\newcommand{\bias}[0]{b}

\newcommand{\activation}[0]{\varrho}

\newcommand{\NNvec}[0]{\mathrm{x}}

\newcommand{\norm}[1]{\left\lVert#1\right\rVert}

\newcommand{\ScaleConst}{c_{\textup{sc}}}
\newcommand{\onestepoperator }{R}
\newcommand{\onestepoperatorVTWO }{S} 

\title[Complexity of NNs to solve structured linear systems]{Complexity bounds on neural networks for the solution of structured linear systems of equations}
\author[B.~D\"orich, R.~Maier, L.~Ullmer]{Benjamin D\"orich${}^*$, Roland Maier${}^*$, Lukas Ullmer${}^*$}
\address{${}^*$Institute for Applied and 
Numerical Mathematics, Karlsruhe Institute of Technology, Englerstr.~2, 76131 Karlsruhe, Germany}
\email{\{benjamin.doerich,roland.maier\}@kit.edu}

\begin{document}

\begin{abstract}
We derive upper bounds on the complexity of ReLU neural networks approximating the solution of a linear system given the matrix and the right-hand side. 
We focus on matrices which are symmetric positive definite and sparse, as they appear in the context of finite difference and finite element methods. 
For such matrices, we extend available results for the matrix inversion to the task of solving a linear system, where we leverage favorable properties of classical methods such as the modified Richardson and the conjugate gradient method.
Our bounds on the number of layers and neurons are not only explicit with respect to the size of the matrices, but also with respect to their condition numbers. 
\end{abstract} 



\maketitle

\section{Introduction}

In this work, we investigate the complexity of neural networks for the approximate solution of linear systems of the form
\begin{equation} \label{eq:LGS_intro}
    A x = \rhs,
\end{equation}
where $A \in \Reals^{n \times n}$, $n \in \NatNum$ is a (possibly large) matrix with a certain structure, $\rhs\in \Reals^n$ a given vector, and $x \in \Reals^n$ the (unknown) solution to the system. Specifically, we have in mind linear systems that arise in the context of finite element or finite difference methods.
These have the property that $A$ is typically sparse with a known sparsity structure, symmetric, and in general have a large condition number. The main goal is to understand the realization of the mapping $(A,\rhs) \mapsto x = A^{-1} \rhs$ via classical feed-forward neural networks. While we do not aim at practically realizing such a construction, the main goal is to understand the complexity of a possible realization and, particularly, its dependence on the condition number and the sparsity pattern. In essence, this shall provide a rough estimation on the necessary size of a neural network to realize the solution of a linear system of the form~\eqref{eq:LGS_intro} for a specific class of matrices. 

The study of approximation properties and sufficient complexity bounds on neural networks for different tasks is an active field of research. First theoretical results on approximation capabilities for certain classes of functions have been presented in~\cite{Cyb89,HorSW89} but without explicit dependencies on the size of the networks. An approximation result with explicit rates has first been obtained in~\cite{Bar93} and important findings on the approximation of a scalar multiplication operator based on the polarization formula are presented in~\cite{Yar17}, which have been employed to approximate smooth functions with ReLU neural networks in the $L^\infty$-norm. These findings can be extended to approximation results in Sobolev norms, see~\cite{GueKP20}. 
In~\cite{PetV18}, a framework for studying approximation properties of neural networks has been presented that is based on the idea of using smaller neural networks as building blocks for certain, more complicated, approximation tasks. It has been used to show approximation results in the context of, e.g., higher-order finite elements~\cite{OpsJPS20,OpsS24} and parametrized/high-dimensional partial differential equations~\cite{ElbGJS21,KutPRS22,MarS23} (see also~\cite{GeiPRSK21} for a corresponding numerical study on the complexity).
An overview on the expressivity of neural networks can be found in~\cite{GuhRK23}. 

The main goal of this work is to approximate the solution process of a potentially sparse linear systems of equations by mimicking iterative algorithms. In~\cite{KutPRS22}, the inversion of a matrix is approximated up to arbitrary precision by a neural network based on the Neumann series. Further, upper bounds on the size of the neural network are presented in terms of depth and number of non-zero neurons. While this approach could be directly used to construct a neural network that implements the solution of a linear system, the overall complexity scales worse than $n^3$ for solving a system with~$n$ unknowns, which even exceeds the effort for inverting the matrix directly. 
Note that the main motivation in~\cite{KutPRS22} are rather small matrices and uniformly bounded condition numbers in the context of a reduced basis method, where the cubic dependence is not severe. 
In contrast, we have large matrices in mind, where such a complexity is prohibitively expensive.
Particularly for sparse linear systems in the context of, e.g., finite element methods, the complexity bounds appear suboptimal, as well-known classical iterative methods exist that do not require the inversion of the matrix and can even solve such systems with almost linear complexity. 
Note that the dominant cubic scaling in~\cite{KutPRS22} can be improved to $n^{\log_2 7} \approx n^{2.8}$ as recently proposed in~\cite{RomBP26*}. This is possible by constructing the networks for the matrix-matrix multiplication via the \emph{Strassen algorithm} 
\cite{Str69} and implementing those in the Neumann series. 
For sparse linear systems, this is still suboptimal and to the best of our knowledge, there are no other works studying the complexity analysis of the solution of linear systems of equations specifically.

In this work, we aim at constructing a neural network that realizes the solution of a linear system of equations and which makes use of iterative algorithms, particularly for sparse systems. Exemplarily, we construct  neural networks that realize the
modified Richardson and the cg-algorithm. 
Compared to the mentioned works, we pay special attention to the role of the condition number and, in particular, prove our estimates with an explicit tracing of the condition number. One of the main findings is that the complexity suffers from large condition numbers, mainly since the underlying iterative methods deteriorate in this case. Note, however, that our approach opens up the possibility to study the combination of iterative methods with pre-conditioners. This can enable the construction of upper bounds on the necessary size of neural networks which are less sensitive to large condition numbers.  

Besides the interesting question of how the (worst possible) condition of a matrix influences the construction of a neural network and in particular its size, the considerations of this work may also be applied in the context of more involved constructions to understand the complexity of neural networks for specific tasks. For instance, in~\cite{KroMP23}, the complexity estimate in~\cite{KutPRS22} directly enters in the context of replacing bottleneck computations in finite element-based multiscale methods by appropriate neural networks. The suboptimal scaling of the size of required neural networks therein is automatically and substantially improved with our construction. 

\subsection*{Outline of the paper}
We first introduce our notation in Section~\ref{sec:notation} and present for reference the important result on the approximate matrix inversion established in \cite{KutPRS22} adapted to our setting. 
In Section~\ref{sec:results}, we present the overall framework as well as our main results. Further,
we discuss the application to stiffness matrices from a finite element background
and the implications on other works that use neural networks in their algorithms. 
Several useful results, which are used throughout the analysis, are recalled in Section~\ref{sec:preps} and the building block for our main theorems are shown.
Sections~\ref{sec:Ric_proof} and \ref{sec:cg_proof} are devoted to the proofs of our main results employing the Richardson and the cg-approach, respectively. 

\section{Notation}
\label{sec:notation}

In this section, we fix the notation of several spaces and objects which are used throughout this work. 
First, we note that the dimension $n$ of the system in \eqref{eq:LGS_intro} is fixed throughout the paper. 
We denote the space of admissible matrices and vectors appearing in \eqref{eq:LGS_intro} by
\begin{equation} \label{def:matrix_size}
    \Knz{z}
    \coloneqq \{  A \in \Reals^{n\times n} \mid \|A\|_2 \leq z 
    \},
    \qquad
    \knz{z} \coloneqq \{  \rhs \in \Reals^{n} \mid \|\rhs\|_2 \leq z  \},
\end{equation}
for some $z>0$,
where $\|\cdot\|_2$ denotes the standard Euclidean norm for vectors or the spectral norm for matrices, respectively.
In addition, we need the class of symmetric, positive definite matrices with bounded spectrum, which we denote by
\begin{equation} \label{def:matrix_size}
    \Sigman{\lambda,\Lambda}
    \coloneqq \{  A \in \Reals^{n \times n} \mid A = A^\top, \lambda \leq \mineig(A) \leq \maxeig(A) \leq \Lambda 
    \},
\end{equation}
where $\mineig(A)$ and $\maxeig(A)$ denote the smallest and largest eigenvalue of a matrix~$A$, respectively. 
We denote the condition number of $A$ by $\kappa(A)$ and recall that for any $A \in \Sigman{\lambda,\Lambda}$ we have the uniform bound 
\begin{equation} \label{eq:uppercondbound}
 \kappa(A) \leq \frac{\Lambda}{\lambda}  \eqqcolon\uppercondbound.  
\end{equation}
This bound is used for $\alpha \in \{\frac12, 1\}$ in the context of the function
\begin{equation} \label{eq:def_rho_alpha}
    \rho_\alpha = \rho_\alpha (\uppercondbound) = \frac{\uppercondbound^\alpha - 1}{\uppercondbound^\alpha + 1} \in [0,1),
\end{equation}
which will later relate to the convergence rate of our methods.

We further introduce the space of matrices with a special sparsity pattern. We define the list  $\numNZrowvector = (\numNZrowvector_1,\ldots,\numNZrowvector_n)$, where 
$\numNZrowvector_i \in \{1,\ldots,n\}^{\numNZrowi{i}}$ are vectors of length $\numNZrowi{i}$, respectively. $\numNZrowvector_i$ encodes the columns in the $i$th row of a matrix, where entries are allowed to be different from zero. In particular, if $j$ is  \emph{not} an element of $\numNZrowvector_i$, then the entry $A_{i,j} = 0$ (but the reverse statement is not necessarily true).
Matrices of a sparsity pattern contained in $\numNZrowvector$ are collected in the set
\begin{equation}\label{def:sparse_pattern}
    \begin{aligned} 
    \SparseSpacekk 
    =
    \{
    A \in \Reals^{n \times n}
    &\text{
    \text{s.t. the sparsity pattern of $A$ is contained in $\numNZrowvector$}
      }
    \} ,
\end{aligned}
\end{equation}
and set $\numNZ \coloneqq \sum_{i=1}^n \numNZrowi{i}$ as the maximal number of non-zero elements, and
note that 
$n \leq \numNZ 
\leq n^2$.
Further, we define 
\begin{equation}
    \numNZrowMax = \max_{i=1,\ldots,n} \numNZrowi{i}.
\end{equation}
For the application we have in mind, it is reasonable to assume an a-priori known super-set of the sparsity pattern, and we discuss this after our main results. In addition, we discuss possible extension in Section~\ref{sec:extensions} below.

\begin{remark}
Within this work, we assume that the matrix $A$ is given in a COO-type format, i.e., there is a vector $A^\textup{v}$, which contains the $\numNZ$ values and two vectors which contain the corresponding row and column index, respectively. However, since we have a fixed sparsity pattern, the two vectors for the row and column indices do not change and can be neglected.

In particular, this allows us to abuse the notation as follows: we will interchangeably use the matrix $A$ and the vector with its (relevant) values $A^\textup{v}$. That is, a matrix~$A$ as input of a neural network (cf.~Definition~\ref{def:neuralnet} below) makes sense due to its (compressed) vector representation. Multiple matrices and/or vectors as inputs have to be understood as a long vector containing all the input values. 
\end{remark}

Next, we present the type of neural network which we consider in this work.
\begin{definition} \label{def:neuralnet}
Let $L \in \NatNum $. 
A neural network $\NNs$ of depth $\layers$ consists of a list of
matrix-vector-tupels
\begin{equation}
    \NNs = \left[(\weight_1,\bias_1),(\weight_2,\bias_2),\ldots,(\weight_\layers,\bias_\layers)\right], 
\end{equation}
where the weights $\weight_\ell$
are a matrix in $ \Reals^{N_\ell \times N_{\ell-1}}$ 
and the corresponding bias $\bias_\ell$
is a vector in $ \Reals^{N_\ell}$
with $N_0,\ldots,N_\layers \in \NatNum$. 
We call $ \mathrm{dim}_{\mathrm{in}}(\NNs):=N_0 $
the input dimension and
$ \mathrm{dim}_{\mathrm{out}}(\NNs):=N_\layers $ the output dimension. 
Together with the ReLU activation function $\activation \colon  \Reals 
\to  \Reals$, $x \mapsto \max\{0,x\}$ 
and the input vector $\NNvec$,
the neural network is recursively defined via 
\begin{align}
    \NNvec_0 &:= \NNvec, \\
    \NNvec_\ell &:= \activation(\weight_\ell\NNvec_{\ell-1} + \bias_\ell) \quad \text{for } \ell = 1,\ldots, \layers-1,  \\
    \NNvec_\layers &:= \weight_\layers\NNvec_{\layers-1}+\bias_\layers,
\end{align}
where $\activation$ acts component-wise on vectors by convention.
Overall, for a given input~$\NNvec_0$ the output $\NNvec_\layers$
of $\NNs$ is given by the relation
\begin{equation}
    \NNs(\NNvec_0)= \NNvec_\layers
\end{equation}
and the neural network $\NNs$ can be understood as a map 
$ \Reals^{N_0} \to  \Reals^{N_\layers}$. 
The network is said to have $\layers$ layers
and $N_j$ is the number of neurons in the $j$-th layer.
With $\norm{A}_0$ the number of nun-zero entries of $A$, 
for $\ell\leq \layers$ we denote by 
\begin{equation} 
    \entries_\ell:=\norm{\weight_\ell}_0+\norm{\bias_\ell}_0
\end{equation}
the number of weights in the $\ell$-th layer
and by  
\begin{equation}
    \entries:=\sum^\layers_{\ell=1}\entries_\ell
\end{equation}
the total number of weights of $\NNs$. 
In addition, we use the functions $\layersfun$ and $\entriesfun$ to assign a neural network $\NNs$ its number of layers and weights, i.e.,
$\layersfun(\NNs) = \layers$ and $\entriesfun(\NNs) = \entries$.
\end{definition}

As a reference result, we present the one from the seminal work \cite{KutPRS22}, where the authors construct a neural network for a general matrix $A \in \Reals^{n\times n}$
without any sparsity pattern and under the assumption of a uniformly bounded condition number.
As an important reference, we recall the essential parts of the result using our notation introduced above.

\begin{theorem}[{cf.~\cite[Thm.~3.8]{KutPRS22}}]\label{thm:kutyniok}
Let $\delta \in (0,1)$. 
Then, for any $\epsilon\in(0,\frac{1}{4})$ there exists a neural network $\NNs^{\textup{inv}}$ such that
\begin{equation}
   \sup\limits_{A\in \Knz{1-\delta}}
                \norm{\bigl(\Id-A\bigr)^{-1}-\NNs^{\textup{inv}}\bigl(A)}_2 
                \leq 
                \epsilon,
\end{equation}
where $\mathrm{dim}_{\mathrm{in}}(\NNs^{\textup{inv}})=n^2$ and $\mathrm{dim}_{\mathrm{out}}(\NNs^{\textup{inv}})=n^2$. 
Further, with
\begin{equation}
    m_{\textup{N}} =m_{\textup{N}}(\epsilon,\delta):= \left\lceil\frac{\log_2(0.5\epsilon\delta)}{\log_2(1-\delta)}\right\rceil,
\end{equation}
there exists a constant $C_\textup{inv} >0$ which is independent of $m_{\textup{N}}$, $n$, $\epsilon$, and $\delta$ such that
\begin{itemize}
\item[(i)] $\layersfun\Bigl(\NNs^{\textup{inv}} \Bigr)
            \leq
            C_\textup{inv}  \log_2(m_{\textup{N}})
            \Bigl( \log_2(\frac{1}{\epsilon})+\log_2(m_{\textup{N}})+\log_2(n)\Bigr)$,
\item[(ii)] $\entriesfun\Bigl( \NNs^{\textup{inv}}\Bigr) 
             \leq  
             C_\textup{inv}  m_{\textup{N}} \log^2_2(m_{\textup{N}}) 
             \Bigl( \log_2(\frac{1}{\epsilon})+\log_2(m_{\textup{N}})+\log_2(n)\Bigr) n^3$.
\end{itemize}
\end{theorem}

The result in Theorem~\ref{thm:kutyniok} is the starting point for our investigations. In the upcoming section, we aim to exploit the knowledge of a certain pattern that matrices in~$\SparseSpacekk$ 
have; cf.~\eqref{def:sparse_pattern}. The goal is to reduce the computational complexity (in the sense of number of layers and neurons) also for matrices with a large condition number. In particular, we aim for a better scaling of the number of weights in the dimension $n$.

\section{Main results}
\label{sec:results}

In this section, we first present our main results and apply them to a class of matrices as they usually appear in the context of finite element or finite difference methods. In particular, we are interested in symmetric and positive definite matrices with specific sparsity patterns. 
As an example of special interest and based on the notation in~\eqref{def:sparse_pattern}, we will mainly consider the case
\begin{equation}
\numNZ =  \mathcal{O}(n),
\quad
\kappa(A) \gg 1.
\end{equation}
In the context of finite elements,
$\numNZrowi{i}$ represents the number of neighboring nodes of the $i$th node $a_i$ in the mesh, which is a uniformly bounded number even under refinement. This directly implies the scaling of $\numNZ$. 
However, we could also account for systems with linear constraints such as the Poisson problem with Neumann boundary conditions, where the last row consists of $\mathcal{O}(n)$ non-zero elements. 
Since we aim at solving linear systems of the form~\eqref{eq:LGS_intro}, i.e.,
\begin{equation}
    A x = \rhs,
\end{equation}
an important aspect of our work is that we do not approximate the map $A \mapsto A^{-1}$, but rather $(A,\rhs) \mapsto x=A^{-1}\rhs$, which can be realized by an adaptation of well-known iterative methods.

\subsection{Abstract results} 

We first present a result which resembles the modified Richardson iteration, and is thus closely related to the result from~\cite{KutPRS22}. In addition, we employ a cg-type approach by constructing the optimal polynomial related to the required number of iteration steps to obtain a given tolerance. 

\subsubsection*{Richardson method}
For our first result, we briefly revisit the modified Richardson method \cite{Ric1910}. 
Starting with the original system in \eqref{eq:LGS_intro},
we rewrite the problem for $A \in \Sigmanll$ as a fixed-point problem
\begin{equation}  \label{eq:Richardson_iter}
    x = (I - \omega A) x + \omega \rhs
    \quad
    \text{with}
    \quad 
    \omega = \frac{2}{\lambda + \Lambda}.
\end{equation}
For an eigenvalue $\lambda_i$ of $A$, we observe with $\uppercondbound$ defined in \eqref{eq:uppercondbound} that
\begin{equation}
    1 - \omega \lambda_i = \frac{\Lambda - 2 \lambda_i + \lambda}{\lambda + \Lambda} = \frac{\uppercondbound - 2 \frac{\lambda_i}{\lambda}+ 1}{\uppercondbound + 1}.
\end{equation}
Since $1 \leq \frac{\lambda_i}{\lambda} \leq \uppercondbound$, the definition of 
$\rho_1$ in  \eqref{eq:def_rho_alpha} 
gives $I - \omega A \in \Knz{\rho_1}$.
By assumption it holds $\rho_1 < 1$, and we have convergence due to the Banach fixed-point theorem.
In addition, we obtain for $x^0 = 0$ the explicit formula
\begin{equation} \label{eq:Neumann}
    x^{m+1} = \sum\limits_{\ell = 0}^{m} (I - \omega A)^\ell (\omega \rhs).
\end{equation}
Thus, the construction is mainly based on two important steps:
First, we pre-compute $\widehat{A} = I -\omega A$ and $\widehat{\rhs} = \omega \rhs$,
and second,
 approximate the Neumann series~\eqref{eq:Neumann} with a neural network, adapting the results in \cite{KutPRS22}. 
    More precisely, the matrix-matrix products are
    replaced by matrix-vector products, which immediately reduces the complexity of the network.
    Further, we reduce the effort by taking the sparsity pattern of $A$ into account.
With this strategy, we are able to prove the following result. 

\begin{theorem}[Approximation via modified Richardson iteration]\label{thm:Richardson}

Let $0<\lambda<\Lambda$ and let $\omega$ be given as in \eqref{eq:Richardson_iter} and $\rho_1$
as in \eqref{eq:def_rho_alpha}.
Then, for any 
$\epsilon\in(0,1)$ 
and 
$\ScaleConst \in [1,\frac{1+ \uppercondbound}{2 }]$ 
there exists a neural network $\NNs^{\textup{Ric}}$ such that
\begin{equation}
       \sup\limits_{A\in \SparseSpacekk \cap  \Sigman{\lambda,\Lambda}, \, \rhs \in \knz{\ScaleConst \lambda} } 
        \norm{ A^{-1} \rhs -\NNs^{\textup{Ric}}\bigl(A,\rhs\bigr)}_2 
            \leq 
            \epsilon,
\end{equation}
where $\mathrm{dim}_{\mathrm{in}}(\NNs^{\textup{Ric}})=n(n+1)$ and $\mathrm{dim}_{\mathrm{out}}(\NNs^{\textup{Ric}})=n$. 
Further, we define
\begin{equation}  
m_{\textup{Ric}}  
=
m_{\textup{Ric}}(\epsilon,\ScaleConst,\rho_1)
=
\left\lceil 
\frac{ | \log_2 ( \frac{\epsilon}{2 \ScaleConst} ) | }{| \log_2(\rho_{1})|}
\right\rceil.
\end{equation}
Then there exists a constant
$C_\textup{Ric} > 0$ independent of $m_{\textup{Ric}}$, $n$, $\eta$, and $\epsilon$, such that 

\begin{itemize}
\item[(i)] $
 \layersfun\bigl(\NNs^{\textup{Ric}} \bigr)
 \leq
C_\textup{Ric}
m_\textup{Ric}    
\bigl(  \log_2 (\frac{1}{\epsilon})  
+
\log_2(n)
+
\log_2 (m_\textup{Ric}) 
                   \bigr),
             $
\item[(ii)]
$
\entriesfun\bigl(\NNs^{\textup{Ric}}\bigr) 
    \leq  
   C_\textup{Ric}
m_\textup{Ric}  
    \bigl(  \log_2 (\frac{1}{\epsilon})  +  \log_2(n)
+
 \log_2 (m_\textup{Ric}) 
        \bigr)
    \numNZ .
    $     
\end{itemize} 
\end{theorem}

\begin{proof}
The proof is postponed to Section~\ref{sec:Ric_proof}.
\end{proof}

\subsubsection*{cg-type method}
Next, we consider the approximation via a cg-type approach. To this end, 
we rescale $\widehat{A} = \frac{1}{\Lambda} A \in \Knz{1}$ and
$\widehat{\rhs} = \frac{1}{\Lambda} \rhs\in \knz{1} $ for $\rhs\in \knz{\Lambda}$.
For the construction, we decompose the approximation error as    
\begin{equation} \label{eq:cg_decomp}
    A^{-1}\rhs- \NNs^{\textup{cg} }(A,r)
=
\bigl( A^{-1}\rhs- q_{m-1}(\widehat{A}) \widehat{\rhs}  \bigr)
+
\bigl( q_{m-1}(\widehat{A}) \widehat{\rhs}  - \NNs^{\textup{cg}} (A,r)\bigr),
\end{equation}
with an appropriate polynomial $q_{,m-1} \in \calP_{m-1}$, where $\calP_{m-1}$ denotes the space of polynomials up to degree $m-1$.  
The first term on the right-hand side of~\eqref{eq:cg_decomp} is estimated by classical numerical linear algebra results to obtain a suitable parameter $m_{\textup{cg}}   = m_{\textup{cg}}  (\epsilon)$.
The network $\NNs^{\textup{cg}}$ is then constructed in such a way that it approximates $q_{m-1}(\widehat{A}) \widehat{\rhs}$ 
up to the desired tolerance. This allows us to prove the following result.

\begin{theorem}[Approximation via cg-type iteration]\label{thm:cg}

Let $0<\lambda<\Lambda$ and let $\omega$ be given as in \eqref{eq:Neumann} and $\rho_1$
as in \eqref{eq:def_rho_alpha}.
Then, for any 
$\epsilon\in(0,1)$ 
and 
$\ScaleConst \in [1,\uppercondbound]$, 
there exists a neural network $\NNs^{\textup{cg}}$ such that 
\begin{equation}
       \sup\limits_{A\in \SparseSpacekk \cap  \Sigman{\lambda,\Lambda}, \, \rhs\in \knz{\ScaleConst \lambda} } 
        \norm{ A^{-1} \rhs-\NNs^{\textup{cg}}\bigl(A,\rhs\bigr)}_2 
            \leq 
            \epsilon,
\end{equation}
where $\mathrm{dim}_{\mathrm{in}}(\NNs^{\textup{cg}})=n(n+1)$ and $\mathrm{dim}_{\mathrm{out}}(\NNs^{\textup{cg}})=n$. 
Further, define
\begin{equation}  
m_{\textup{cg}}  
=
 m_{\textup{cg}} (\epsilon,\ScaleConst,\rho_{1/2})
=
\left\lceil 
\frac{ | \log_2 ( \frac{\epsilon}{4 \ScaleConst} ) | }{| \log_2(\rho_{1/2})|}
\right\rceil.
\end{equation}
Then, there exists a constant
$C_\textup{cg} > 0$ independent of $m_{\textup{cg}}$, $n$, $\eta$, and $\epsilon$, such that
\begin{itemize}
\item[(i)] $
 \layersfun\bigl(\NNs^{\textup{cg}} \bigr)
 \leq
C_\textup{cg}
m_\textup{cg}    
\bigl(  \log_2 (\frac{1}{\epsilon})  
+
\log_2(n)
+
\log_2 (m_\textup{cg}) 
                   \bigr),
             $
\item[(ii)]
$
\entriesfun\bigl(\NNs^{\textup{cg}}\bigr) 
    \leq  
   C_\textup{cg}
m_\textup{cg}  
    \bigl(  \log_2 (\frac{1}{\epsilon})  +  \log_2(n)
+
 \log_2 (m) 
        \bigr)
    \numNZ.
    $     
\end{itemize} 
\end{theorem}

\begin{proof}
The proof is postponed to Section~\ref{sec:cg_proof}.
\end{proof}

\begin{remark}
The range of $\ScaleConst$ in Theorems~\ref{thm:Richardson} and~\ref{thm:cg} ensures that the rescaled version of the right-hand side fulfills $\norm{\widehat{r}} \leq 1$, which simplifies several computations in the proof. 
To be precise, we observe for the rescaling $\widehat{r} = \omega r$ in  Theorem~\ref{thm:Richardson} that
$\lambda \ScaleConst \leq \omega^{-1}$
and for
$\widehat{r} = \frac{1}{\Lambda} r$ in  Theorem~\ref{thm:cg}
we have the relation
$\lambda \ScaleConst  \leq  \Lambda$.
Larger right-hand sides could be treated as well, which would lead to additional (logarithmic) terms in the complexity bounds, but we do not further discuss this here for the sake of readability. 
\end{remark}

\begin{remark} \label{rem:scaling_consid}
For finite element methods, the influence of the condition number (through $\rho_\alpha$ in~\eqref{eq:def_rho_alpha}) is important in determining the number of required iterations when using iterative algorithms. Therefore, the asymptotic scaling of terms involving $\rho_\alpha$ is essential.

We have
$ \frac{1}{|\log_2 \rho_\alpha|}$ in the denominator of both $m_{\textup{Ric}}$ and $m_{\textup{cg}}$, which is the main difference in the size of both networks.
Using the definition of $\rho_\alpha$, we obtain
\begin{equation}
\rho_\alpha^{-1}  =1 + \tfrac{2}{\uppercondbound^\alpha - 1}
\end{equation}
and thus by the Taylor approximation of the logarithm
\begin{equation}
    |\log_2 \rho_\alpha|^{-1}
    =
    |\log_2 \bigl( 1 + \tfrac{2}{\uppercondbound^\alpha - 1}
    \bigr)|^{-1}
     \lesssim  
     \uppercondbound^\alpha
\end{equation}
for $\uppercondbound$ sufficiently large. Thus, we can see that for large values of $\uppercondbound$, the cg-type network can be chosen to be much
smaller than the Richardson-type one.
\end{remark}

\begin{remark} \label{rem:comp_with_full}
To compare our result with Theorem~\ref{thm:kutyniok}, we note that in the derivation in Section~2 of \cite{KutPRS22}, the matrix $A$ stems from a rescaled system, and the parameter $\delta$ is approximately given by $\frac{\mineig(A)}{\maxeig(A)} = \frac{1}{\kappa} \in (0,1)$.
In particular, we again observe the scaling
\begin{equation}
    | \log_2 (1-\delta) |^{-1} 
    =
     | \log_2 ( 1 + \tfrac{1}{\kappa - 1}) |^{-1}
     \sim 
     \kappa,
\end{equation}
which resembles the scaling of our Richardson approach. However, let us emphasize that in \cite{KutPRS22} the network was applied to matrices which are ``pre-conditioned'' in a dimension-reduction step, such that therein the condition numbers are considered moderate.
\end{remark}

\subsection{Application to finite element methods}\label{ss:appFEM}

In the following, we discuss our results for the special case of matrices stemming from a finite element discretization of some partial differential equation $\textup{L} u = f$ 
for some second-order elliptic differential operator $\textup{L}$,
right-hand side $f \in L^2(\Omega)$, and 
bounded domain $\Omega$. 

\subsubsection*{Condition number}

We denote the mesh width by $h$, which is the largest diameter over all elements. 
It is well-known that for quasi-uniform meshes, the scaling of the eigenvalues implies
\begin{equation}
    \lambda \sim h^{d}, \quad 
    \Lambda
    \sim h^{d-2} .
\end{equation}
Further, we note the relation $n \sim h^{-d}$ of the degrees of freedom and the mesh size, which allows us to deduce the scaling 
\begin{equation}
    \uppercondbound \sim h^{-2}
    \sim n^{2 /d},
\end{equation}
for $\uppercondbound$ defined in \eqref{eq:uppercondbound}.

\subsubsection*{Band width}

Focusing on regular meshes, we denote by $N$ the number of degrees of freedom per spatial direction, i.e., we set $N = n^{1/d}$. In the simplest case, the sparsity pattern can be chosen to be a band of the matrix. For one-dimensional problems this allows to choose 
$\numNZrowi{i} = 3$, and thus $\numNZ \leq 3n$,
in dimension two 
$\numNZrowi{i} \sim N$,
and thus 
$\numNZ \lesssim n N \sim n^{3/2} $
and in dimension three 
$\numNZrowi{i} \sim N^2$,
and thus 
$\numNZ \lesssim n N^2 \sim n^{5/3} $.

Including more knowledge on the mesh structure, one can in principle also exploit that there is only a moderate number of non-zero entries per row, allowing $\numNZrowi{i} \leq C$ for a uniform constant $C$ depending on the maximal number of incident elements at each node. If this information is available for the whole class of considered problems, one can also obtain $\numNZ \lesssim n$, which is the typical case for finite element methods.

\subsubsection*{Load vector}

Assembling the load vector $\rhs$ is performed by computing the integrals
\begin{equation}
   \rhs_i =  \int_\Omega f \varphi_i \,dx
\end{equation}
where $\varphi_i$ are the standard ansatz function of the finite element method. Denoting the support of $\varphi_i$ by $S_i = \textup{supp} \, \varphi_i$, we obtain for $f\in L^p(\Omega)$, $p\geq 2$ and conjugate $q = p/(p-1)\leq 2$,
\begin{equation}
 |\rhs_i| = | \int_{S_i} f \varphi_i \,dx |
 \leq
 \norm{f}_{L^p(S_i)}  \norm{\varphi_i}_{L^q(S_i)} 
  \lesssim
 \norm{f}_{L^p(S_i)}
 h^{d/q} .
\end{equation}
Hence, we have
\begin{equation}
    \norm{\rhs}_2 
    \leq
    n^{1/2-1/p}
    \norm{\rhs}_p
    =
    n^{1/2-1/p}
      \big( \sum\limits_{i=1}^n
     |\rhs_i|^p
     \bigr)^{1/p}
         \lesssim
     n^{1/2-1/p}
     h^{d/q}
       \norm{f}_{L^p(\Omega)},
\end{equation}
and by the relation $n = h^{-d}$ and $\frac{1}{q} = \frac{p-1}{p}$, we obtain
\begin{equation}
    \norm{b}_2 
    \lesssim
     h^{d/2}
   \norm{f}_{L^p(\Omega)} 
   \lesssim h^{-d/2} \lambda
    \norm{f}_{L^p(\Omega)} .
\end{equation}
In view of Theorem~\ref{thm:Richardson} and Theorem~\ref{thm:cg}, this implies that $\ScaleConst \sim h^{-d/2} \sim n^{2/d}$, and thus the iteration numbers $m_{\textup{Ric}}$ and $m_{\textup{cg}}$ include an additional logarithmic scaling with respect to the degrees of freedom~$n$.

\subsubsection*{Implication for constructed networks in the literature}
In the specific setting that not directly the inversion of the matrix is of interest but rather the solution of a linear system, our results directly improve the complexity bounds that immediately follow from an application of the results in~\cite{KutPRS22}, see Theorem~\ref{thm:kutyniok} above. Note that this is not a surprise as the results in~\cite{KutPRS22} are much more general and the main motivation therein is the solution of smaller systems in the context of a reduced basis approach. Our results are particularly advantageous in the case of large sparse matrices with a prescribed sparsity pattern, which is relevant in the context of, e.g., finite element or finite difference methods. 
In that regard, our results improve the complexity bounds for the construction in~\cite{KroMP23}. Therein, suitable (worst-case) complexity bounds are derived for the neural network realization of the multiscale method known as \emph{localized orthogonal decomposition}, see also~\cite{KroMP22} and~\cite{MalP20,AltHP21} for an overview on the original methodology. The construction sets up stiffness matrices by using locally constructed neural networks on certain patches. More precisely, on a coarse mesh on the scale of interest $H > 0$ with $N = H^{-d}$ degrees of freedom a finite element-type method is constructed that takes into account fine-scale features from a much finer scale $h < H$. This comes at the cost of a slightly increased sparsity pattern with $|\log_2 (H)|$ entries per row. While~\cite[Thm.~4.3]{KroMP23} suggests a scaling of the number of layers and overall neurons of $\calO(|\log_2 (h)|^2)$ and $\calO(h^{-2}|\log_2(h)|^4(|\log_2(H)|H/h)^{3d})$, respectively, by using the result in~\cite{KutPRS22} (see Theorem~\ref{thm:kutyniok} above), our construction in, e.g., Theorem~\ref{thm:cg} works with only $\calO(h^{-1}|\log_2(h)|(|\log_2(H)|^2H/h)^d )$ neurons. The construction is deeper, though, requiring $\calO(h^{-1}|\log_2 (h)|^2)$ layers. With appropriate pre-conditioning, one may potentially even avoid the pre-factor~$h^{-1}$ for an almost optimal scaling (up to logarithmic terms), see also the discussion in Section~\ref{sec:extensions}. 

\section{Preliminaries for the proofs}
\label{sec:preps}

The main goal of this section is to establish the construction and analysis of networks which can realize elementary maps that are essential for the two iterative methods.  
First, we consider the map used for the Richardson iteration given by
\begin{equation}\label{eq:defOneStepOp}
\onestepoperator \colon  \Reals^{\numNZ} \times \Reals^n \times \Reals^n
  \to
  \Reals^\numNZ \times \Reals^n \times \Reals^n
  \quad
    (A,r,c) \mapsto (A,Ar, r + c) ,
\end{equation}
which will be the cornerstone of the construction of the network in Theorem~\ref{thm:Richardson}. 

\begin{theorem} \label{thm:one_step_affine}
Let $\epsilon \in (0,1)$. Then, there exists a neural network $\NNs^\onestepoperator$ 
with 
$\dim_\mathrm{in}(\NNs^\onestepoperator) = \numNZ + 2n$
and
$\dim_\mathrm{out}(\NNs^\onestepoperator) = \numNZ + 2n$ 
such that for $z \geq 1$
\begin{equation}
    \sup_{A \in \SparseSpacekk \cap \Knz{1},\,
    r \in \knz{z},\,
    c \in \Reals^n}
    \| \onestepoperator (A,r,c) - \NNs^\onestepoperator(A,r,c) \|_2 \leq \epsilon
\end{equation}
Further, there exists a generic constant $C_\onestepoperator > 0$ independent of $n$, $A$, $r$, $\epsilon$, $\eta$, and $c$ 
such that
\begin{itemize}
\item[(i)] $
 \layersfun\Bigl(\NNs^{\onestepoperator} \Bigr)
 \leq
C_{\onestepoperator}
             \Bigl( \log_2(\frac{1}{\epsilon})
             +
             \log_2(n)
             +
             {
             \log_2(z)
             }
            \Bigr),
             $
\item[(ii)]
$
\entriesfun\Bigl(\NNs^{\onestepoperator}\Bigr) 
    \leq  
    C_{\onestepoperator}
    \Bigl( \log_2(\frac{1}{\epsilon})
    +
    \log_2(n)
    +
    {
    \log_2(z)
    }
    \Bigr)
    \numNZ $.       
\end{itemize} 
\end{theorem}

The proof of the theorem is given at the end of this section. Second, for the cg-type iteration we construct the map
\begin{equation}\label{eq:defOneStepOp_2} 
\onestepoperatorVTWO_\alpha\colon  \Reals^{\numNZ} \times \Reals^n \times \Reals^n
  \to
  \Reals^\numNZ \times \Reals^n \times \Reals^n,
  \quad
    (A,r,c) \mapsto (A,\alpha\mathbf{1} + 2 Ar - c , r)
\end{equation}
for some fixed  $\alpha \leq 1$ and $\mathbf{1} \coloneqq (1,\ldots,1)^\top$. 
This map is the basis to perform a neural network version of the Clenshaw algorithm and thus to prove 
Theorem~\ref{thm:cg}.

\begin{theorem} \label{thm:one_step_affine2}
Let $\epsilon \in (0,1)$. Then, there exists a neural network $\NNs^{\onestepoperatorVTWO_\alpha}$ 
with 
$\dim_\mathrm{in}(\NNs^{\onestepoperatorVTWO_\alpha}) = \numNZ + 2n$
and
$\dim_\mathrm{out}(\NNs^{\onestepoperatorVTWO_\alpha}) = \numNZ + 2n$ 
such that for $z \geq  1$
\begin{equation}
    \sup_{A \in \SparseSpacekk \cap \Knz{1},\,
    r \in \knz{z},\,
    c \in \Reals^n}
    \| \onestepoperatorVTWO_\alpha (A,r,c) - \NNs^{\onestepoperatorVTWO_\alpha}(A,r,c) \|_2 \leq \epsilon
\end{equation}
Further, there exists a generic constant $C_{\onestepoperatorVTWO} > 0$ independent of $n$, $A$, $r$, $c$, $\epsilon$, $\eta$, and $\alpha$ such that
\begin{itemize}
\item[(i)] $
 \layersfun\Bigl(\NNs^{\onestepoperatorVTWO_\alpha} \Bigr)
 \leq
C_{\onestepoperatorVTWO}
             \Bigl( \log_2(\frac{1}{\epsilon})
             +
             \log_2(n)
             + 
             \log_2(z)
            \Bigr),
             $
\item[(ii)]
$
\entriesfun\Bigl(\NNs^{\onestepoperatorVTWO_\alpha}\Bigr) 
    \leq  
    C_{\onestepoperatorVTWO}
    \Bigl( \log_2(\frac{1}{\epsilon})
    +
    \log_2(n)
    + 
     \log_2(z)
    \Bigr)
    \numNZ.
    $     
\end{itemize} 
\end{theorem}

The proof is also provided at the end of this section.
Note that one could as well use an adaption of Theorem~\ref{thm:one_step_affine} for the \Chebname polynomials,
replacing the third entry of $\onestepoperator$ defined in~\eqref{eq:defOneStepOp} by
$\alpha r +c$.
However, this results in the multiplication with different large coefficients $\alpha$ up to size $\sim 2^m$ and thus induces additional factors of~$m$ in the scaling of the layers and neurons. 

We now provide all the auxiliary results to eventually prove the two theorems on the approximation of~$\onestepoperator$ and~$\onestepoperatorVTWO_\alpha$, starting with the identity network.

\begin{lemma}[identity network, cf.~Lem.~2.3 and Rem.~2.4 in~\cite{PetV18}]
\label{lem:identity}
For any $L \in \NatNum$, $L \geq 2$, there exists a neural network $\NNs^\mathrm{id}$ that realizes the identity in $\Reals^k$ with $\dim_\mathrm{in}(\NNs^\mathrm{id}) = \dim_\mathrm{out}(\NNs^\mathrm{id}) = k$ 
\begin{itemize}
    \item[(i)] $\layersfun\bigl(\NNs^\mathrm{id}\bigr)
                = L$,
    \item[(ii)] $\entriesfun\bigl( \NNs^\mathrm{id}\bigr) 
                 \leq  
                 2kL$.
\end{itemize}
\end{lemma}

The identity network is an important ingredient to define appropriate (sparse) concatenations of neural networks. A main issue with a classical concatenation of two neural networks $\NNs^1$ and $\NNs^2$ (denoted $\NNs^1 \bullet \NNs^2$) is that it does not easily allow for an estimation of $M(\NNs^1 \bullet \NNs^2)$ which depends linearly on $M(\NNs^1)$ and $M(\NNs^2)$ (due to the merge of two layers from different networks). This issue can be resolved by the introduction of an intermediate identity network.

\begin{definition}[sparse concatenation]\label{def:sparseconc}
Let the networks
\begin{equation}
    \NNs^1 = \left[(\weight^1_1,\bias^1_1),\ldots,(\weight^1_{\layers_1},\bias^1_{\layers_1})\right], \quad \NNs^2 = \left[(\weight^2_1,\bias^2_1),\ldots,(\weight^2_{\layers_2},\bias^2_{\layers_2})\right] 
\end{equation}
with $\dim_\mathrm{in}(\NNs^1) = \dim_\mathrm{out}(\NNs^2)$ be given. Further, denote with $\NNs^\mathrm{id}$ the identity network with input and output dimension $\dim_\mathrm{in}(\NNs^1) = \dim_\mathrm{out}(\NNs^2)$. Then we call
\begin{equation}
    \NNs^1 \circ \NNs^2 \coloneqq \NNs^1 \bullet \NNs^\mathrm{id} \bullet \NNs^2
\end{equation}
the sparse concatenation of $\NNs^1$ and $\NNs^2$. 
\end{definition}

\begin{definition}[parallelization, cf.~\cite{PetV18,ElbGJS21}]\label{def:par}
    Let $\NNs^1,\ldots,\NNs^j$ be neural networks with $\layers$ layers each. Then the neural network
    \begin{align}
        P(\NNs^1,&\ldots,\NNs^j) \\&\coloneqq \left[\left(\begin{pmatrix} \weight_1^1 & &\\&\ddots&\\&&\weight_1^j\end{pmatrix},\begin{pmatrix} \bias_1^1\\\vdots\\\bias_1^j \end{pmatrix}\right),\ldots,\left(\begin{pmatrix} \weight_\layers^1 & &\\&\ddots&\\&&\weight_\layers^j\end{pmatrix},\begin{pmatrix} \bias_\layers^1\\\vdots\\\bias_\layers^j \end{pmatrix}\right) \right]
    \end{align}
    realizes the parallelization of $\NNs^1,\ldots,\NNs^j$ (without inter-network communication). The definition naturally extends to networks that do not have the same number of layers by bridging missing layers with identity networks, which will be the case in the following.
\end{definition}

The next result follows directly from a repeated (recursive) application of~\cite[Lem.~5.3]{ElbGJS21}. Note that the result is typically not sharp. 

\begin{lemma}[sparse concatenation]
Let $\NNs^1,\ldots,\NNs^j$ be $j$ neural networks with $\dim_\mathrm{in}(\NNs^i) = \dim_\mathrm{out}(\NNs^{i+1})$, $i = 1,\ldots, j-1$. Then the sparse concatenation $\NNs \coloneqq \NNs^1 \circ \cdots \circ \NNs^j$ as defined in Definition~\ref{def:sparseconc} satisfies
 \begin{itemize}
\item[(i)] $\layersfun\bigl(\NNs\bigr)
            \leq
            \sum_{i=1}^j\layersfun\bigl(\NNs^i\bigr)$, 
\item[(ii)] 
$\entriesfun\bigl( \NNs\bigr) 
             \leq  
             3\sum_{i=1}^j\entriesfun\bigl(\NNs^i\bigr)$. 
\end{itemize}
\end{lemma}

We also have a result on the parallelization of networks, see~\cite[Lem.~5.4]{ElbGJS21}.

\begin{lemma}[parallelization] 
\label{lem:parallelization}
Let $\NNs^1,\ldots,\NNs^j$ be neural networks and define the network $\NNs \coloneqq P(\NNs^1,\ldots,\NNs^j)$ as their parallelization according to Definition~\ref{def:par}. Then 
\begin{itemize}
\item[(i)] $\layersfun\bigl(\NNs \bigr)
            \leq
            \max_{i = 1,\ldots,j} \layersfun\bigl(\NNs^i\bigr)$,
\item[(ii)] $\entriesfun\bigl( \NNs\bigr) 
             \leq  
             2\, \Bigl(\sum_{i=1}^j \entriesfun\bigl(\NNs^i\bigr)\Bigr) + 4\,\Bigl(\sum_{i=1}^j \dim_\mathrm{out}\bigl(\NNs^i\bigr)\Bigr)
             \,\max_{i = 1, \ldots,j} \layersfun \bigl(\NNs^i\bigr)$.
\end{itemize}
\end{lemma}

The following lemma follows directly from~\cite[Prop.~3.7]{KutPRS22} and is adjusted for our purposes as the special case of a matrix-matrix product.

\begin{proposition}[Scalar product] \label{prop:scalar_product}
Let $\epsilon \in (0,1)$ and $k \in \NatNum$. 
There exists a neural network $\NNs^{\mathrm{sca},k}$ such that  for $z\geq 1$
\begin{equation}
   \sup_{
   x,y \in \Reals^k, \norm{x}_2 \leq 1, \norm{y}_2 \leq z
   }  |y^\top x - \NNs^{\mathrm{sca},k}(y,x)| \leq \epsilon .
\end{equation}
Further, there exists a generic constant $C_\mathrm{sca} > 0$ independent of $k$, $y$, $x$, $\epsilon$, and $z$ such that 
\begin{itemize}
\item[(i)] $\layersfun\bigl(\NNs^{\mathrm{sca},k} \bigr)
            \leq
            C_\textup{sca}
             \Bigl( \log_2(\frac{1}{\epsilon})+\log_2(k)
             +  
             {\log_2(z)}
            \Bigr)$,
\item[(ii)] $\entriesfun\bigl( \NNs^{\textup{sca},k}\bigr) 
             \leq  
             C_\textup{sca}  
             \Bigl( \log_2(\frac{1}{\epsilon})+\log_2(k)
             + 
             {\log_2(z)}
             \Bigr) k$.
\end{itemize}
\end{proposition}

With Proposition~\ref{prop:scalar_product}, we may deduce the generalization to matrix-vector multiplication for sparse matrices, which is the essential ingredient for approximating the maps~$\onestepoperator$ and~$\onestepoperatorVTWO_\alpha$. 

\begin{theorem}[sparse matrix-vector multiplication]
\label{thm:matrix_vec_mult}
Let $\epsilon \in (0,1)$. Then there exists a neural network $\NNs^\mathrm{mult}$
with $\dim_\mathrm{in}(\NNs^\mathrm{mult}) = \numNZ + n$ and $\dim_\mathrm{out}(\NNs^\mathrm{mult}) = n$ such that 
for  $z\geq 1$
\begin{equation}
    \sup_{A \in \SparseSpacekk \cap \Knz{1},\,r \in \knz{z}} \| A r - \NNs^\mathrm{mult}(A,r) \|_2 \leq \epsilon .
\end{equation}
Further, there exists a generic constant $C_\mathrm{mult} > 0$ independent of $n,\,A,\, x$, $\epsilon$, and $z$ such that
\begin{itemize}
\item[(i)] $
 \layersfun\Bigl(\NNs^{\mathrm{mult}} \Bigr)
 \leq
C_\textup{mult}
             \Bigl( \log_2(\frac{1}{\epsilon})
             +
             \log_2(n)
             +
             {\log_2(z)}
            \Bigr)
             $,
\item[(ii)]
$
\entriesfun\Bigl(\NNs^{\mathrm{mult}}\Bigr) 
    \leq  
    C_\textup{mult} 
    \Bigl( \log_2(\frac{1}{\epsilon})
    +
    \log_2(n)
    +
    {\log_2(z)}
    \Bigr)
    \numNZ
    $.
\end{itemize}
\end{theorem}

\begin{proof}
    Recall that the exact sparsity pattern of $A$ is given. For the $i$th row, we first restrict $r$ to the relevant entries with a one-layer network with $\calO(\numNZrowi{i})$ non-zero entries (a classical restriction matrix). Then, we approximate the sparse scalar product of the $\numNZrowi{i}$ entries of $e_i^\top A$, where $e_i$ denotes the $i$-th unit vector, with the corresponding restricted vector of $r$-values, which altogether realizes $(Ar)_i$. With Proposition~\ref{prop:scalar_product}, this is possible up to accuracy $\epsilon/\sqrt{n}$ with a network $\NNs^{\mathrm{sca},\numNZrowi{i}}$ fulfilling
    \begin{align}
        \layersfun\Bigl(\NNs^{\mathrm{sca},\numNZrowi{i}} \Bigr)
                    &\leq
                    C_\textup{sca}
                     \Bigl( \log_2(\tfrac{1}{\epsilon})+\log_2(n)
                     + {\log_2(z)}
                    \Bigr),
                    \\
                            \entriesfun\Bigl( \NNs^{\textup{sca},\numNZrowi{i}}\Bigr) 
                     &\leq  
                     C_\textup{sca}  
                     \Bigl( \log_2(\tfrac{1}{\epsilon})+\log_2(n)
                     + {\log_2(z)}
                     \Bigr) \numNZrowi{i}.
    \end{align}
    Note that we have used that $\numNZrowi{i} \leq n$. 

    Computing the $(Ar)_i$ in parallel, we obtain with Lemma~\ref{lem:parallelization} for the final network $\NNs^\mathrm{mult}$ the desired bounds, using that $\numNZ = \sum_{i=1}^n \numNZrowi{i}$ and 
    \begin{equation}\label{eq:boundMult}
        \| A r - \NNs^\mathrm{mult}(A,r) \|_2^2 = \sum_{i=1}^n |(Ar)_i - \NNs^{\mathrm{sca},\numNZrowi{i}}(e_i^\top A,r)|^2 \leq n \epsilon^2/n = \epsilon^2.
    \end{equation}
    Note that we slightly misuse the notation by implicitly treating $e_i^\top A$ and $r$ as vectors of length $\numNZrowi{i}$. 
\end{proof}

The next result is again important for the construction of both $\onestepoperator$ and $\onestepoperatorVTWO_\alpha$.

\begin{lemma} \label{lem:scaling}
    Let $\alpha \in \Reals$ and $n \in \NatNum$. Then, there exists a neural network $\NNs^\mathrm{scale}$, that realizes the map 
    \begin{equation}
        \Reals^n \times \Reals^n \to \Reals^n,
        \quad
        (x,y) \mapsto \alpha x + y.
    \end{equation}
    Further, we have 
    \begin{itemize}
        \item[(i)] $\layersfun\Bigl(\NNs^{\mathrm{scale}} \Bigr)
                    \leq
                    2$,
        \item[(ii)] $\entriesfun\Bigl( \NNs^{\textup{scale}}\Bigr) 
                     \leq  
                     8 n$.
        \end{itemize}
\end{lemma}

\begin{proof} 
    The result follows directly from using Lemma~\ref{lem:identity} with $L=2$ for $x$ (scaled by the factor $\alpha$) and $y$ in parallel, and a summation in the last layer. 
\end{proof}

Finally, we conclude this section by assembling all the ingredients for the proofs of Theorem~\ref{thm:one_step_affine} and Theorem~\ref{thm:one_step_affine2}. 

\begin{proof}[Proof of Theorem~\ref{thm:one_step_affine}]
For the construction of the network to approximate the operator $\onestepoperator$, we parallelize the networks that 
\begin{itemize}
    \item realize the identity with input $A$ based on Lemma~\ref{lem:identity},
    \item approximate $Ar$ with accuracy $\epsilon$ according to Theorem~\ref{thm:matrix_vec_mult},
    \item calculate $ r + c$ based on Lemma~\ref{lem:scaling}. 
\end{itemize}
Clearly, the approximation of $Ar$ requires the deepest network.
Therefore, we directly get with Lemma~\ref{lem:parallelization} and a slight adjustments of the constants in Theorem~\ref{thm:matrix_vec_mult}
\begin{equation}
    \layersfun\Bigl(\NNs^{\onestepoperator} \Bigr) \leq  C_\onestepoperator
                     \Bigl( \log_2(\tfrac{1}{\epsilon})
                     +
                     \log_2(n)
                     + 
                    {\log_2(z)}
                    \Bigr)
\end{equation}
and
\begin{equation}
    \entriesfun\Bigl(\NNs^{\onestepoperator}\Bigr) 
            \leq  
            C_\onestepoperator 
            \Bigl( \log_2(\tfrac{1}{\epsilon})
            +
            \log_2(n)
            +
            {\log_2(z)}
            \Bigr)
            \numNZ. 
\end{equation}
Note that we have used that $n \leq \numNZ$. Therefore, the effort for computing the approximation of $Ar$ dominates the overall complexity. 
This concludes the proof. 
\end{proof}

\begin{proof}[Proof of Theorem~\ref{thm:one_step_affine2}]
The construction is very similar to the one in the proof of Theorem~\ref{thm:one_step_affine}. We parallelize the networks that
\begin{itemize}
    \item realize the identity with input $A$ using Lemma~\ref{lem:identity},
    \item approximate $Ar$ with accuracy $\epsilon$ according to Theorem~\ref{thm:matrix_vec_mult} followed by applications of Lemma~\ref{lem:scaling}. 
    \item realize the identity with input $r$ using once more Lemma~\ref{lem:identity},
\end{itemize}
Again, the dominant contribution comes from the approximation of $Ar$ and with Lemma~\ref{lem:parallelization} the bounds from Theorem~\ref{thm:one_step_affine} are resembled (up to constants).
\end{proof}

\section{Richardson method - proof of Theorem~\ref{thm:Richardson}}
\label{sec:Ric_proof}

In this section, we provide the construction of the network stated in Theorem~\ref{thm:Richardson}. In the first step, we determine the necessary number of iterations~$m$
in \eqref{eq:Neumann} to achieve the required accuracy.

\begin{lemma} \label{lem:pol_approx_r}
Let $\epsilon >0$, $\ScaleConst\geq 1$,
and let $\omega$ be given as in \eqref{eq:Neumann}. 
Then, there exists an integer $m\geq 1$ such that
\begin{equation}
   \sup\limits_{A\in \SparseSpacekk \cap  \Sigman{\lambda,\Lambda}, \, \rhs\in \knz{\ScaleConst \lambda} } 
    \Bigl\|A^{-1} \rhs
    -
    \sum\limits_{\ell = 0}^{m} (I - \omega A)^\ell (\omega \rhs)
   \Bigr\|_2 
        \leq 
        \frac{\epsilon}{2},
\end{equation}
where $m = m(\epsilon,\ScaleConst,\rho_1)$
satisfies
\begin{equation}  
m \geq \frac{ | \log_2 ( \frac{\epsilon}{2 \ScaleConst} ) | }{| \log_2(\rho_{1})|}
\end{equation}
with 
$\rho_{1}$ defined in \eqref{eq:def_rho_alpha}.
\end{lemma}

\begin{proof}
We use the same technique as for the error bound of the Richardson method,
see for example
\cite[Sec.~4.2.1]{Saa03_book},
and recall the approximation sequence from~\eqref{eq:Neumann} 
with $x^0 = 0$. Using that $x =A^{-1}r = (\omega A)^{-1} \omega r$ and defining the error
$e^m = x -x^m$, we obtain the recursion
\begin{equation}
    e^m = (I - \omega A) e^{m-1}
     = (I - \omega A)^m x .
\end{equation}
Since $\| I - \omega A \|_2 \leq \rho_1$, we conclude
\begin{equation}
\|e^m\|_2 
\leq
\rho_1^m \|x\|_2
\leq
\rho_1^m \|A^{-1}\|_2 \|r\|_2
\leq
\rho_1^m \lambda^{-1} \ScaleConst \lambda
=
\rho_1^m \ScaleConst .
\end{equation}
We can then guarantee the claimed bound if
\begin{equation}
\rho_1^m \ScaleConst \leq  \frac{\epsilon}{2}
\qquad\text{or, equivalently,}\qquad 
(\rho_1)^{-m} 
 \geq \frac{2 \ScaleConst}{\epsilon}
\end{equation}
as stated.
\end{proof}

Finally, we provide the construction of the neural network which approximates the polynomial in \eqref{eq:Neumann}.

\begin{proof}[Proof of Theorem~\ref{thm:Richardson}]
The main idea of the proof is to approximate the polynomial construction in Lemma~\ref{lem:pol_approx_r} with the optimal $m = m_\mathrm{Ric}$, followed by a triangle inequality.

In the first step, we construct a network to perform the rescaling
$\widehat{A} = I -\omega A$ and $\widehat{\rhs} = \omega \rhs$,
using Lemma~\ref{lem:scaling}. The complexity for this operation scales linearly and is therefore negligible for the overall complexity bounds. The same statement holds true for the final rescaling in the very end.

We focus on the main scaling, which is introduced via the realization of the polynomial in Lemma~\ref{lem:pol_approx_r}. We express the desired polynomial via the map of Theorem~\ref{thm:one_step_affine}, and note that
\begin{equation}\label{eq:onestepm}
\onestepoperator^{k}(X^0) = \onestepoperator \circ \ldots \circ \onestepoperator (X^0)
= \Bigl( \widehat{A}, \widehat{A}^{m} \widehat{r}, \sum_{\ell=0}^{k-1} \widehat{A}^\ell \widehat{r}\Bigr)   
,
\quad
X^0 = (\widehat{A},\widehat{r},0),
\end{equation}
where we also use $\circ$ here for the standard concatenation of operators.
We approximate~\eqref{eq:onestepm} with $k-1$ sparse concatenations of $\NNs^{\onestepoperator}$ with itself and abbreviate
$[\NNs^\onestepoperator]^k = \NNs^\onestepoperator \circ \ldots \circ \NNs^\onestepoperator$. Note that $\NNs^\onestepoperator$ is 
constructed with 
accuracy $\delta = \delta(\epsilon) = {\epsilon}/(2m^2) > 0$ 
for input variables 
of size $z \leq m +2$.

(a) We now show by induction that for $i = 0,\ldots,m$ 
\begin{equation} \label{eq:induction_hyp_Ric}
    \| \onestepoperator^i (X^0) -[\NNs^\onestepoperator]^i (X^0)  \|_2 \leq \tfrac\epsilon2,
    \quad
    \| [\NNs^\onestepoperator]^i  (X^0)  \|_2 \leq i + 3 ,
\end{equation}
and note that the induction is trivially anchored at $i = 0$, if we define in this case both operators as the identity.

Now let $1 \leq i \leq m$, assume that \eqref{eq:induction_hyp_Ric} holds for $0,\ldots,i-1$, 
and compute
\begin{align}
     \| \onestepoperator^i (X^0) -& [\NNs^\onestepoperator]^i (X^0) \|_2
    \\
    &\leq
    \sum\limits_{\ell=0}^{i-1}
     \| \onestepoperator^{i-\ell} ( [\NNs^\onestepoperator]^{\ell} (X^0)) 
     -  \onestepoperator^{i-\ell-1} ( [\NNs^\onestepoperator]^{\ell+1} (X^0)) \|_2
    \\
    &=
    \sum\limits_{\ell=0}^{i-1}
        \| \onestepoperator^{i-\ell-1} (\onestepoperator( [\NNs^\onestepoperator]^{\ell} (X^0))) 
     -  \onestepoperator^{i-\ell-1} (\NNs^\onestepoperator ( [\NNs^\onestepoperator]^{\ell} (X^0))) \|_2
    \\
     &=
    \sum\limits_{\ell=0}^{i-1}
     \| \onestepoperator^{i-\ell-1} (\onestepoperator (X^\ell)) 
     - \onestepoperator^{i-\ell-1}( \NNs^\onestepoperator  (X^\ell)) \|_2
\end{align}
for 
$X^{\ell} =[\NNs^\onestepoperator]^{\ell} (X^0)$. 
Further,
note that the first and last argument of $\onestepoperator ( X^{\ell})$ and $\NNs^\onestepoperator ( X^{\ell})$ are identical,
and we denote the second arguments by
$b$ and $\widetilde{b}$, respectively. We now have to consider the stability of $\onestepoperator^{k}$ for such arguments.
We first compute
\begin{align}
\| \onestepoperator^{k} (\widehat{A},b,c) - \onestepoperator^{k}   (\widehat{A},\widetilde{b},c) \|_2 
&=
\Bigl\| ( 0 , \widehat{A}^k (b-\widetilde{b} ),  \sum\limits_{j=0}^{k-1} \widehat{A}^j (b-\widetilde{b} )  \Bigr\|_2
\leq (k+1) \| b-\widetilde{b} \|_2 ,
\end{align}
using that $\|\widehat{A}\|\leq 1$.
Since $\ell \leq i-1$, we can apply \eqref{eq:induction_hyp_Ric} to obtain the bounds on the argument $X^{\ell}$ and we can use the approximation property in
Theorem~\ref{thm:one_step_affine} to obtain 
$\| b-\widetilde{b} \|_2\leq \delta$. Plugging this into the sum and using $i-1 \leq m$, we estimate
\begin{align}
\| \onestepoperator^i (X^0) - [\NNs^\onestepoperator]^i (X^0) \|_2
    \leq
  m^2 \delta =  \tfrac\epsilon2 .
\end{align}
For the second part of the induction, we write 
\begin{align}
    \| [\NNs^\onestepoperator]^i (X^0) \|_2
&\leq
    \| \onestepoperator^i (X^0)  \|_2
+
   \| \onestepoperator^i (X^0) - [\NNs^\onestepoperator]^i (X^0) \|_2
  \leq
  i + 2 + \tfrac\epsilon2 \leq  i + 3,
\end{align}
using $\|\widehat{r}\|\leq 1$ and the stability of $\onestepoperator^i$ based on~\eqref{eq:onestepm}. This closes the induction.

(b) For the size of the network, we first observe that with $m = m_\textup{Ric}$
\begin{equation}
    \log_2 (\tfrac{1}{\delta}) \sim    \log_2 (\tfrac{1}{\epsilon})  +   \log_2 (m_\textup{Ric}) 
\end{equation}
and the estimates in Theorem~\ref{thm:one_step_affine} then gives
\begin{itemize}
\item[(i)] $
 \layersfun\bigl(\NNs^{\textup{Ric}} \bigr)
 \leq
C_{\textup{Ric}}
m_\textup{Ric}
 \bigl(  \log_2 (\frac{1}{\epsilon})  +   
             \log_2(n) 
             +
             \log_2 (m_\textup{Ric}) 
            \bigr),
             $
\item[(ii)]
$
\entriesfun\bigl(\NNs^{\textup{Ric}}\bigr) 
    \leq  
     C_{\textup{Ric}}
m_\textup{Ric}
    \bigl(  \log_2 (\frac{1}{\epsilon})  
    +
         \log_2(n)
             +   \log_2 (m_\textup{Ric}) 
            \bigr)
    \numNZ,
    $     
\end{itemize} 
which is precisely the claimed network size.

Finally, we combine Lemma~\ref{lem:pol_approx_r} with the above construction, which gives the result using the triangle inequality. 
\end{proof}

\section{Cg-type method - proof of Theorem~\ref{thm:cg}}

\label{sec:cg_proof}

In this section, we prove our second main result in Theorem~\ref{thm:cg},
for which we recall the decomposition in \eqref{eq:cg_decomp}. The first estimate can be deduced from classical estimates, which we collect in the following lemma, see for example~\cite{Saa03_book}.
\begin{lemma} \label{lem:pol_approx_cg}
Let $\epsilon >0$. Then, there exists a polynomial $q_{m-1} \in \calP_{m-1}$ such that 
\begin{equation}
   \sup\limits_{A\in \SparseSpacekk \cap  \Sigman{\lambda,\Lambda}, \, \rhs\in \knz{\ScaleConst \lambda} } 
    \norm{ A^{-1} \rhs
    -
    q_{m-1}(\tfrac{1}{\Lambda}  A ) 
    \tfrac{1}{\Lambda}\rhs
   }_2 
        \leq 
        \frac{\epsilon}{2}
\end{equation}
where $m = m(\epsilon,\ScaleConst,\rho_{1/2})$
satisfies
\begin{equation}  
m \geq \frac{ | \log_2 ( \frac{\epsilon}{4 \ScaleConst} ) | }{| \log_2(\rho_{1/2})|}
\end{equation}
with 
$\rho_{1/2}$ defined in 
\eqref{eq:def_rho_alpha}.
\end{lemma}

\begin{proof}
The idea is to start with the (matrix) polynomial which is used in order to prove the approximation property of the cg method.
Let us denote 
$\widehat{A} = \tfrac{1}{\Lambda}  A$
and
$\widehat{\rhs} = \tfrac{1}{\Lambda} \rhs$.
For the convergence proof of the cg method, 
one solves the minimization problem over all polynomial $p_m \in \calP_m$ with $p_m(0)=1$ such that the induced energy norm
\begin{equation}
    \| p_m(\widehat{A} ) 
       \widehat{\rhs}  \|_{\widehat{A}}
\end{equation}
becomes small. Here, the optimal choice is the rescaled \Chebname polynomial given as 
\begin{align}
    p_m(z) = \frac{T_m(\sigma(z))}{T_m(\sigma(0))}, \quad
    \sigma(z) = \frac{\uppercondbound+1}{\uppercondbound-1} - \frac{2}{\Lambda - \lambda} z.
\end{align}
Since $p_m(0) = 1$, we can define the polynomial of degree $m-1$ by 
$q_{m-1}(z) = \frac{p_m(z)-1}{z}$.
This gives for $x_m = q_{m-1}(\tfrac{1}{\Lambda}  A ) 
\tfrac{1}{\Lambda}\rhs$ the estimate in the induced energy norm, i.e.,
\begin{equation}
    \| x - x_m \|_{\widehat{A}} \leq 2 \rho_{1/2}^m \|x\|_{\widehat{A}} ,
\end{equation}
see \cite[Sec.~6.11.2]{Saa03_book}.
This allows us to conclude 
\begin{align}
    \|  x- x_m\|_2 &\leq 
    \mineig^{-1/2}(\widehat{A}) \|  x- x_m\|_{\widehat{A}} 
    \leq 
    \mineig^{-1/2}(\widehat{A}) 2 \rho_{1/2}^m \|\widehat{A}^{-1/2} \widehat{A} x\|_2
    \\
    &\leq
    \mineig^{-1}(\widehat{A}) 2 \rho_{1/2}^m 
    \|\widehat{\rhs}\|_2
    = 
    \mineig^{-1}
    (A) 2 \rho_{1/2}^m 
    \|\rhs\|_2
    \leq 2 \ScaleConst \rho_{1/2}^m ,
\end{align}
and hence we can ensure
$ \|  x- x_m\|_2 \leq \frac{\epsilon}{2} $ by the above choice of $m$.
\end{proof}

Before we eventually present the proof of Theorem~\ref{thm:cg}, we need to recall some results on \Chebname polynomials. Besides the ones of first kind, which we denoted by $T_j$, we further need the \Chebname polynomials of second kind, denoted by $U_j$. They are defined via the same three-term recurrence, starting however with $U_0(x) = 1$ and $U_1(x) = 2x$. 
The crucial identity for our construction is 
given by 
\begin{equation} \label{eq:rep_Tm_diff}
    \frac{T_m(x) - T_m(x_0)}{x-x_0} = 
    \sum\limits_{\ell=0}^{m-1} \alpha_\ell(x_0)  U_\ell (x)  ,
\end{equation}
with the coefficients 
\begin{equation}
\alpha_\ell (x_0) = 
\begin{cases}
    T_{m-1-\ell}(x_0) & \text{for $\ell = 0$ and $\ell = m-1$},
    \\
    2 T_{m-1-\ell}(x_0) & \text{for $1 \leq \ell \leq m -2$},
\end{cases}
\end{equation}
see \cite[Thm.~5]{KimL16}.
%
We can thus rewrite the polynomial $q_{m-1}$ as
\begin{align}
    q_{m-1}(z) 
    &=
  \frac{ T_m(\sigma(z)) - T_m(\sigma(0)) } {T_m(\sigma(0)) z}
   =
   - \frac{2}{(\Lambda- \lambda ) T_m(\sigma(0))}
  \frac{ T_m(\sigma(z)) - T_m(\sigma(0)) } {  \sigma(z) - \sigma(0)},
\end{align}
such that we can plug in \eqref{eq:rep_Tm_diff} for $x=\sigma (z)$ and $x_0=\sigma(0)$.
Further, if we define $$\alpha_{\max} = \max\limits_{\ell=0,\ldots,m-1} \alpha_\ell(\sigma(0)),$$
the main effort in the network can be traced back to the construction of an approximation of 
\begin{align} \label{eq:cheb_sum_scaled}
    \sum\limits_{\ell=0}^{m-1} 
    {\frac{
    \alpha_\ell(\sigma(0))
    }{\alpha_{\max}}}
     U_\ell (\sigma(\widehat{A})) 
    \widehat{r},
\end{align}
which is a \Chebname series that can be evaluated by the Clenshaw algorithm~\cite{Cle55} 
with $m$~matrix-vector multiplications. Indeed, for $p(x) = \sum_{\ell=0}^{m-1} \alpha_\ell U_\ell(x)$, the iteration reads
\begin{equation}
    b_k = \alpha_k + 2x b_{k+1} - b_{k+2}, \quad
    k=m-1,\ldots,0,
    \quad
    \text{with}
    \quad
    b_m = b_{m+1} = 0,
\end{equation}
and the result is given by $p(x) = b_0$. This is realized by 
the concatenation of the maps 
$\onestepoperatorVTWO_{\alpha_k}$ defined in
\eqref{eq:defOneStepOp_2},
and approximated by the neural networks constructed in 
Theorem~\ref{thm:one_step_affine2}.

\begin{proof}[Proof of Theorem~\ref{thm:cg}]

The proof is similar to the one of Theorem~\ref{thm:Richardson}. We approximate the construction in Lemma~\ref{lem:pol_approx_cg} for $m= m_\mathrm{cg}$ by a neural network and then apply the triangle inequality. 
As before, we first need to perform the rescaling of 
$\widehat{A} = \tfrac{1}{\Lambda}  A$
and $ \widehat{\rhs} = \tfrac{1}{\Lambda} \rhs$. Further, we need to compute
$\sigma(\widehat{A})$. Again, this is done in linear complexity and can thus be neglected. After computing 
\eqref{eq:cheb_sum_scaled}, we need also need to rescale in linear complexity. 

Thus, we focus on \eqref{eq:cheb_sum_scaled}. Note that the desired polynomial can be expressed via
\begin{equation}\label{eq:chainS}
\onestepoperatorVTWO_0 \circ \ldots \circ \onestepoperatorVTWO_{m-1} (X^0) , \quad X^0 = (\widehat{A},0,0),
\end{equation}
where $\onestepoperatorVTWO_k \coloneqq \onestepoperatorVTWO_{\alpha_k}$, cf.~\eqref{eq:defOneStepOp_2}.
Further, we set $\onestepoperatorVTWO_{i,j} = \onestepoperatorVTWO_{i} \circ \ldots \circ \onestepoperatorVTWO_{j}$  for $j > i$, 
$\onestepoperatorVTWO_{i,j} = \onestepoperatorVTWO_j$ for $i=j$, and $\onestepoperatorVTWO_{i,j} = \mathrm{id}$ the identity map for $j<i$.
We approximate~\eqref{eq:chainS} with the sparse concatenation 
$\NNs^{\onestepoperatorVTWO}_{i,j} = \NNs^\onestepoperatorVTWO_{i} \circ \ldots \circ \NNs^\onestepoperatorVTWO_{j}$, where each $\NNs^\onestepoperatorVTWO_{k}$ approximates~$\onestepoperatorVTWO_k$ as constructed in Theorem~\ref{thm:one_step_affine2} with 
accuracy $\delta = \delta(\epsilon) = {\epsilon}/{(2(m+1)^2)} > 0$ for input variables 
of size $z \leq 3 m^2$. 
Note that we use same convention on $i,j$ for 
$\NNs^{\onestepoperatorVTWO}_{i,j}$ 
as for 
$\onestepoperatorVTWO_{i,j}$.

(a) We now show by induction that for $i = m,\ldots,0$ 
\begin{equation} \label{eq:induction_hyp_cg}
    \| \onestepoperatorVTWO_{i,m-1} (X^0) - \NNs^{\onestepoperatorVTWO}_{i,m-1} (X^0)  \|_2 \leq \tfrac\epsilon2,
    \quad
    \| \NNs^{\onestepoperatorVTWO}_{i,m-1} (X^0)  \|_2 \leq 3 m^2,
\end{equation}
and note that the induction is trivially anchored at $i = m$.

Now let $0 \leq i \leq m-1$, assume that \eqref{eq:induction_hyp_cg} holds for $m,\ldots, i+1$,
and compute
\begin{align}
    \| \onestepoperatorVTWO_{i,m-1} &(X^0) - \NNs^{\onestepoperatorVTWO}_{i,m-1} (X^0) \|_2
    \\
    &\leq
    \sum\limits_{\ell=i}^{m-1}
     \|  \onestepoperatorVTWO_{i,\ell-1} ( \NNs^{\onestepoperatorVTWO}_{\ell,m-1} (X^0)) - \onestepoperatorVTWO_{i,\ell} ( \NNs^{\onestepoperatorVTWO}_{\ell+1,m-1} (X^0)) \|_2
    \\
    &=
    \sum\limits_{\ell=i}^{m-1}
        \| \onestepoperatorVTWO_{i,\ell-1}( \NNs_\ell ( \NNs^{\onestepoperatorVTWO}_{\ell+1,m-1} (X^0))) - \onestepoperatorVTWO_{i,\ell-1}( \onestepoperatorVTWO_\ell  ( \NNs^{\onestepoperatorVTWO}_{\ell+1,m-1} (X^0))) \|_2
    \\
     &=
    \sum\limits_{\ell=i}^{m-1}
     \| \onestepoperatorVTWO_{i,\ell-1} (\NNs_{{\ell}}^\onestepoperatorVTWO  (X^{m-1-\ell})) - \onestepoperatorVTWO_{i,\ell-1} (\onestepoperatorVTWO_{\ell} (X^{m-1-\ell})) \|_2
\end{align}
for 
$X^{m-1-\ell} = \NNs^{\onestepoperatorVTWO}_{\ell+1,m-1} (X^0)$. 
We emphasize that the first and last argument of $\onestepoperatorVTWO_\ell ( X^{m-1-\ell})$ and $ \NNs_\ell^\onestepoperatorVTWO ( X^{m-1-\ell} )$ are identical, and as above  we denote the second arguments by
$b$ and $\widetilde{b}$, respectively.
Thus, we have to consider the stability  of~$\onestepoperatorVTWO_{i,\ell-1}$ for such arguments. 
We first compute for $B =\sigma(\widehat{A})$
\begin{align}
\| \onestepoperatorVTWO_{i,\ell-1}  (B,b,c) - \onestepoperatorVTWO_{i,\ell-1}  (B,\widetilde{b},c) \|_2 
&=
\| ( 0 , U_{\ell-i}(B) (b-\widetilde{b}),   U_{\ell-i-1}(B) (b-\widetilde{b}))  \|_2
 \\
 &
\leq (2 \ell+1) \| b -\widetilde{b} \|_2 ,
\end{align}
with two applications of the estimate $|U_{j}(x)| \leq j+1$ for $x\in[-1,1]$.
Since $i +1 \leq \ell +1 \leq m$, we can apply \eqref{eq:induction_hyp_cg} to obtain the bounds on all the inputs 
$X^{m-1-\ell}$. 
Further, we thus can use the approximation property in
Theorem~\ref{thm:one_step_affine2} to obtain
$\| b -\widetilde{b} \|_2 \leq \delta$
and conclude by the summation over $\ell$ that
\begin{align}
    \| \onestepoperatorVTWO_{i,m-1} (X^0) - \NNs^{\onestepoperatorVTWO}_{i,m-1} (X^0) \|_2
    \leq
  (m+1)^2 \delta =  \tfrac\epsilon2 .
\end{align}
For the second part of the induction, we write
\begin{align}
    \| \NNs^{\onestepoperatorVTWO}_{i,m-1} (X^0) \|_2
&\leq
    \| \onestepoperatorVTWO_{i,m-1} (X^0)  \|_2
+
   \| \onestepoperatorVTWO_{i,m-1} (X^0) - \NNs^{\onestepoperatorVTWO}_{i,m-1} (X^0) \|_2
  \\  &\leq
  m^2 + 1 + \tfrac\epsilon2 \leq  3 m^2,
\end{align}
where we employed the estimate 
\begin{equation} \label{eq:stab_cheb}
    \|  \onestepoperatorVTWO_{i,m-1} (\widehat{A},0,0)  \|_2 \leq m^2 +1.
\end{equation}
The latter inequality follows from
\begin{equation}
\| \onestepoperatorVTWO_{\beta_k} \circ \ldots \circ \onestepoperatorVTWO_{\beta_1} (A,0,0) \|_2
=
\Bigl\|
\Bigl(A ,
\sum\limits_{\ell=1}^{k}
U_{k-\ell}(A) \beta_\ell,
\sum\limits_{\ell=1}^{k-1}
U_{k-\ell-1}(A) \beta_\ell\Bigr)
\Bigr\|_2
\leq k^2 + 1 ,
\end{equation}
for arbitrary coefficients $|\beta_j|\leq 1$,
using once again the estimates on $U_{j}$.
Since we can always choose $k \leq m $, the induction is closed.

(b) For the size of the network, we first observe that
with $m= m_\mathrm{cg}$
\begin{equation}
    \log_2 (\tfrac{1}{\delta}) \sim    \log_2 (\tfrac{1}{\epsilon})  +   \log_2 (m_\mathrm{cg}) 
\end{equation}
and the estimates in Theorem~\ref{thm:one_step_affine2} lead to
\begin{itemize}
\item[(i)] $
 \layersfun\bigl(\NNs^{\textup{cg}} \bigr)
 \leq
C_{\textup{cg}}
m_\textup{cg}
             \bigl(  \log_2 (\frac{1}{\epsilon})  
             +
             \log_2(n)
             +
             \log_2 (m_\mathrm{cg}) 
            \bigr),
             $
\item[(ii)]
$
\entriesfun\bigl(\NNs^{\textup{cg}}\bigr) 
    \leq  
    C_{\textup{cg}}
    m_\textup{cg}
    \bigl(  \log_2 (\frac{1}{\epsilon}) 
    +
    \log_2(n)
    +
    \log_2 (m_\mathrm{cg}) 
            \bigr)
    \numNZ,
    $     
\end{itemize} 
which is precisely the claimed network size.

Finally, we combine Lemma~\ref{lem:pol_approx_cg} with the above construction, which gives the result using the triangle inequality. 
\end{proof}

\section{Possible extensions}
\label{sec:extensions}

In this last section, we discuss three possible extensions of the presented approach. We only provide rough ideas and leave a rigorous treatment for future research. 

\subsubsection*{Pre-processing}
An important assumption of our constructions is that matrices are given with a pre-scribed (maximal) sparsity pattern. As the networks can only process inputs of a fixed size, matrices $A$ with even fewer entries (resulting in a shorter vector $A^\mathrm{v}$ containing the non-zero values) need to be prolonged to fit the maximal sparsity pattern. This requires a classical prolongation matrix, which can be set up with linear complexity. Such a pre-processing step could as well be realized by a neural network, which, however, then depends on the exact positions of the non-zero entries. 

\subsubsection*{Pre-conditioning}

The dependence on condition numbers in classical Richardson or cg iterations directly transfers over to our results. In particular, the number of required iterations (also in the network construction) suffers from large condition numbers, see also the discussion in Section~\ref{ss:appFEM} in the context of finite element-based methods. A natural procedure to avoid the problematic scaling  are pre-conditioning techniques.
To include pre-conditioning into our construction, the necessary operations to set up appropriate well-conditioned systems have to be approximated by neural networks as well. In case that the technique provably lowers the condition number for the classical algorithm, such results would also transfer to corresponding neural network approximations. However, a precise study is beyond the scope of this article.

\subsubsection*{Variable sparsity patterns}

One of the drawbacks of our construction is the necessary a-priori knowledge on the sparsity pattern of the class of matrices. In principle, it is also possible to only keep the number of non-zero entries fixed and derive information about the precise sparsity pattern of a matrix $A$ directly from the input. This requires besides $A^\mathrm{v}$ also information about the matrix entries (as usually given for, e.g., COO-type formats with vectors $A^\mathrm{r}$ and $A^\mathrm{c}$ containing the row and column indices of the non-zero entries, respectively). In order to realize the mapping~$(A,r) \mapsto Ar$ that is relevant for all the presented iterative procedures, for each row of $A$ we need a mapping that reduces the entries of $r$ to the corresponding non-zero entries in the corresponding row of $A$. Realizing such a mapping for all rows and all possible sparsity patterns, however, exceeds the derived complexity bounds. Therefore, we refrain from discussing possible constructions in more detail.

\subsection*{Acknowledgments}
B.~D\"orich and R.~Maier acknowledge funding from the Deut\-sche Forschungsgemeinschaft (DFG, German Research Foundation) -- Project-ID 258734477 -- SFB 1173. 
Parts of this work were conducted during B.~D\"orich's and R.~Maier's stay at the
Hausdorff Research Institute for Mathematics funded by the Deutsche Forschungsgemeinschaft (DFG, German Research Foundation) under Germany's Excellence Strategy – EXC-2047/2 – 390685813.
Preliminary results for this work have been developed in the Master thesis of L.~Ullmer at Karlsruhe Institute of Technology. 
Finally, we thank Matthias Deiml for fruitful discussions. 

\newcommand{\etalchar}[1]{$^{#1}$}


\begin{thebibliography}{GPR{\etalchar{+}}21}

\bibitem[AHP21]{AltHP21}
R.~Altmann, P.~Henning, and D.~Peterseim.
\newblock Numerical homogenization beyond scale separation.
\newblock {\em Acta Numer.}, 30:1--86, 2021.

\bibitem[Bar93]{Bar93}
A.~R. Barron.
\newblock Universal approximation bounds for superpositions of a sigmoidal function.
\newblock {\em IEEE Trans. Inform. Theory}, 39(3):930--945, 1993.

\bibitem[Cle55]{Cle55}
C.~W. Clenshaw.
\newblock A note on the summation of {C}hebyshev series.
\newblock {\em Math. Tables Aids Comput.}, 9:118--120, 1955.

\bibitem[Cyb89]{Cyb89}
G.~Cybenko.
\newblock Approximation by superpositions of a sigmoidal function.
\newblock {\em Math. Control Signals Systems}, 2(4):303--314, 1989.

\bibitem[EGJS22]{ElbGJS21}
D.~Elbr{\"a}chter, P.~Grohs, A.~Jentzen, and C.~Schwab.
\newblock D{NN} expression rate analysis of high-dimensional {PDE}s: application to option pricing.
\newblock {\em Constr. Approx.}, 55(1):3--71, 2022.

\bibitem[GKP20]{GueKP20}
I.~G\"uhring, G.~Kutyniok, and P.~Petersen.
\newblock Error bounds for approximations with deep {R}e{LU} neural networks in {$W^{s,p}$} norms.
\newblock {\em Anal. Appl. (Singap.)}, 18(5):803--859, 2020.

\bibitem[GPR{\etalchar{+}}21]{GeiPRSK21}
M.~Geist, P.~Petersen, M.~Raslan, R.~Schneider, and G.~Kutyniok.
\newblock Numerical solution of the parametric diffusion equation by deep neural networks.
\newblock {\em J. Sci. Comput.}, 88(1):Paper No. 22, 37, 2021.

\bibitem[GRK23]{GuhRK23}
I.~G\"uhring, M.~Raslan, and G.~Kutyniok.
\newblock Expressivity of deep neural networks.
\newblock In {\em Mathematical aspects of deep learning}, pages 149--199. Cambridge Univ. Press, Cambridge, 2023.

\bibitem[HSW89]{HorSW89}
K.~Hornik, M.~Stinchcombe, and H.~White.
\newblock Multilayer feedforward networks are universal approximators.
\newblock {\em Neural Netw.}, 2(5):359--366, 1989.

\bibitem[KL16]{KimL16}
S.~H. Kim and J.~H. Lee.
\newblock On difference quotients of {C}hebyshev polynomials.
\newblock {\em Bull. Korean Math. Soc.}, 53(2):373--386, 2016.

\bibitem[KMP22]{KroMP22}
F.~Kr\"{o}pfl, R.~Maier, and D.~Peterseim.
\newblock Operator compression with deep neural networks.
\newblock {\em Adv. Contin. Discrete Models}, Paper No. 29, 23, 2022.

\bibitem[KMP23]{KroMP23}
F.~Kr\"{o}pfl, R.~Maier, and D.~Peterseim.
\newblock Neural network approximation of coarse-scale surrogates in numerical homogenization.
\newblock {\em Multiscale Model. Simul.}, 21(4):1457--1485, 2023.

\bibitem[KPRS22]{KutPRS22}
G.~Kutyniok, P.~Petersen, M.~Raslan, and R.~Schneider.
\newblock A theoretical analysis of deep neural networks and parametric {PDE}s.
\newblock {\em Constr. Approx.}, 55(1):73--125, 2022.

\bibitem[MP20]{MalP20}
A.~M{\aa}lqvist and D.~Peterseim.
\newblock {\em Numerical homogenization by localized orthogonal decomposition}, volume~5 of {\em SIAM Spotlights}.
\newblock Society for Industrial and Applied Mathematics (SIAM), Philadelphia, PA, 2020.

\bibitem[MS23]{MarS23}
C.~Marcati and C.~Schwab.
\newblock Exponential convergence of deep operator networks for elliptic partial differential equations.
\newblock {\em SIAM J. Numer. Anal.}, 61(3):1513--1545, 2023.

\bibitem[OPS20]{OpsJPS20}
J.~A.~A. Opschoor, P.~C. Petersen, and C.~Schwab.
\newblock Deep {ReLU} networks and high-order finite element methods.
\newblock {\em Anal. Appl.}, 18(05):715--770, 2020.

\bibitem[OS24]{OpsS24}
J.~A.~A. Opschoor and C.~Schwab.
\newblock Deep {R}e{LU} networks and high-order finite element methods {II}: {C}heby\v sev emulation.
\newblock {\em Comput. Math. Appl.}, 169:142--162, 2024.

\bibitem[PV18]{PetV18}
P.~Petersen and F.~Voigtlaender.
\newblock Optimal approximation of piecewise smooth functions using deep {ReLU} neural networks.
\newblock {\em Neural Netw.}, 108:296--330, 2018.

\bibitem[RBP26]{RomBP26*}
G.~Romera and J.~A. Bárcena-Petisco.
\newblock A theoretical analysis on the inversion of matrices via neural networks designed with {S}trassen algorithm.
\newblock {\em ArXiv Preprint}, 2501.06539, 2026.

\bibitem[Ric11]{Ric1910}
L.~F. Richardson.
\newblock {IX.} the approximate arithmetical solution by finite differences of physical problems involving differential equations, with an application to the stresses in a masonry dam.
\newblock {\em Philos. Trans. R. Soc. Lond. A.}, 210(459-470):307--357, 01 1911.

\bibitem[Saa03]{Saa03_book}
Y.~Saad.
\newblock {\em Iterative methods for sparse linear systems}.
\newblock Society for Industrial and Applied Mathematics, Philadelphia, second edition, 2003.

\bibitem[Str69]{Str69}
V.~Strassen.
\newblock Gaussian elimination is not optimal.
\newblock {\em Numer. Math.}, 13:354--356, 1969.

\bibitem[Yar17]{Yar17}
D.~Yarotsky.
\newblock Error bounds for approximations with deep {ReLU} networks.
\newblock {\em Neural Netw.}, 94:103--114, 2017.

\end{thebibliography}
\end{document}